\documentclass[11pt,a4paper]{amsart}
\usepackage{wasowicz_hyperref}
\usepackage{times}

\renewcommand{\cf}{\emph{cf.}}

\info{Manuscript}

\DeclareMathOperator{\Gr}{Gr}

\begin{document}
	\date{November 08, 2016}
	\title[Local affine selections of convex multifunctions]{Local affine selections of convex multifunctions}
	
	\author[S. Wąsowicz]{Szymon Wąsowicz}
	\address{Department of Mathematics, University of Bielsko-Biała, Willowa~2, 43-309 Bielsko-Biała, Poland}
	\email{swasowicz@ath.bielsko.pl}

\begin{abstract}
 It is well known that not every convex multifunction admits an affine selection. One could ask whether there exists at least local affine selection. The answer is positive in the finite-dimensional case. The main part of this note consists of two examples of non-existence of local affine selections of convex multifunctions defined on certain infinite-dimensional Banach spaces.
\end{abstract}

\subjclass[2010]{54C60, 54C65} 

\keywords{multifunction, selection, convexity, extension of a~function, \v{C}ech--Stone compactification}
\maketitle

\section{Introduction}
Given two non-void sets $X$ and $Y$, a map $F\colon X\to 2^Y$ is called a~\emph{multifunction} or a~\emph{set-valued function}. A (single-valued) function $f\colon X\to Y$ is a~\emph{selection} for~$F$, if $f(x)\in F(x)$ for all $x\in X$. There is a plethora of results concerning selections of various kinds, with the Micheal Selection Principle concerning lower semi-continuous selections and the Kuratowski--Ryll-Nardzewski Selection Principle concerning measurable selections as probably the most prominent ones. More recent results connected with Michael Selection Principle were established by Zippin~\cite{Zip90}.\medskip

When $X$ and $Y$ carry a vector-space structure, it is natural to study affine selections or, at least local affine selections for multifunctions $F\colon X\to 2^Y$, which are the objective of this note. This topic was investigated (among others) by A.~Lazar~\cite{Laz68}, A.~Smajdor and W. Smajdor~\cite{SmaSma96}, E.~Behrends and K.~Nikodem~\cite{BerNik95}, M.~Balaj and K.~Nikodem~\cite{BalNik00} and the present author (\cf~\cite{Was95}). 
\par\medskip
We denote by $\nonempty(X)$ the family of all non-empty subsets of a~set~$X$. Now, if $X,Y$ are (real) vector spaces and $D\subset X$ is a~convex set, then the multifunction $F\colon D\to\nonempty(Y)$ is said to be \emph{convex}, if
\begin{equation}\label{eq:conv}
 tF(x)+(1-t)F(y)\subset F\bigl(tx+(1-t)y\bigr)
\end{equation}
for any $x,y\in D$ and $t\in[0,1]$. When the reversed inclusion is stipulated, $F$ is then called \emph{concave}. Of course, the notation $A+B$ and $tA$ is meant in the Minkowski sense, \emph{i.e.}, $A+B=\{a+b\colon a\in A,\;b\in B\}$ and $tA=\{ta\colon a\in A\}$ for any $t\in\R$. Observe that a single-valued function $f\colon D\to Y$ is convex (as a~multifunction, \emph{i.e.}, $f(x)$ is identified with a~singleton $\{f(x)\}$) if and only if $f$ is \emph{affine}, which means that
\[
 tf(x)+(1-t)f(y)=f\bigl(tx+(1-t)y\bigr)\quad \big(x,y\in D,\;t\in[0,1]\big).
\]
It is easy to see that a~multifunction $F$ is convex if and only if its \emph{graph}
\[
 \Gr F=\{(x,y):x\in D,\;y\in F(x)\}
\]
is a~convex subset of $X\times Y$. Moreover, if $F$ is convex, then $F(x)$ is a~convex subset of~$Y$ for any $x\in D$. Indeed, if $y_1,y_2\in F(x)$ and $t\in[0,1]$, then by~\eqref{eq:conv} we get
\[
 ty_1+(1-t)y_2\in tF(x)+(1-t)F(x)\subset F(x).
\]
The condition
\begin{equation}\label{eq:intersection}
 \Bigl(tF(x)+(1-t)F(y)\Bigr)\cap F\bigl(tx+(1-t)y\bigr)\ne\varnothing
\end{equation}
seems to be the weakest one to guarantee the existence of an affine selection for the~multifunction~$F$. Indeed, if $F(x)=\{f(x)\}$, where $f\colon X\to Y$ is affine, the above intersection is a~singleton $\bigl\{f\bigl(tx+(1-t)y\bigr)\bigr\}$.
\par\medskip
It is proved in \cite[Theorem~1]{Was95} that the multifunction~$F$ mapping a~real interval~$\I$ into the family of all compact intervals in~$\R$, admits an affine selection if and only if the condition~\eqref{eq:intersection} is satisfied. In particular, if either $F$ is convex or concave, then~$F$ admits an affine selection.
\par\medskip
One could ask whether a~convex multifunction defined on more general domain admits an affine selection. There is a~number of results going in this direction. One of the versions of the classical Hahn--Banach Separation Theorem guarantees the existence of the linear functional separating two convex subsets of a~topological vector space. It could be easily utilised to prove the existence of a~linear (and hence afine) selection of the certain convex multifunction. Since the problem of extending functions is strongly related to the problem of the existence of selections of multifunctions, we notice that Pełczyński in his PhD dissertation~\cite{Pel68} dealt with linear versions of the classical Tietze--Urysohn theorem (and extended further the classical Borsuk--Dugundji theorem).
\par\medskip
It is worth mentioning that Edwards~\cite{Edw65} proved in 1965 the following separation theorem:
\begin{thm}\label{th:Edwards}
 Let $X$ be a Choquet simplex, $f\colon X\to[-\infty,\infty)$ a~convex upper semicontinuous function and let $g\colon X\to(-\infty,\infty]$ be a~concave lower semicontinuous function such that $f\xle g$ on $X$. Then there exists a~continuous affine function $a\colon X\to\R$ \st\ $f\xle a\xle g$ on~$X$.
\end{thm}
This result, read in the context of multifunctions, states that the lower semicontinuous convex set-valued function defined on a~Choquet simplex, whose values are compact intervals, admits a~continuous affine selection. This multivalued version of Theorem~\ref{th:Edwards} was extended in 1968 by Lazar (\emph{cf.}~\cite[Theorem~3.1]{Laz68}) to more general codomains. 
\begin{thm}\label{th:Lazar}
 Let $\varphi\colon X\to 2^E$ be a~lower semicontinuous affine mapping from a~Choquet simplex $X$ to a Fr\'echet space $E$ that takes non-empty closed values. Then there exists a~continuous affine mapping $h\colon X\to E$ such that $h(x)\in\varphi(x)$ for each $x\in X$.
\end{thm}

In fact, Edwards proved his result in a~form of the necessary and sufficient condition for $X$ to be a Choquet simplex. This means that if a~convex set~$X$ is not a~simplex, one could find two functions $f,g$ (as considered in Theorem~\ref{th:Edwards}), which cannot be separated by the continuous affine function. Hence, in general, a~convex multifunction defined on a~convex subset of a~vector space need not to admit the affine selection. Let us have a look at the well known example due to Olsen~\cite{Ols80} (see also Nikodem~\cite[Remark~1]{Nik89}). Consider the square
\[
 D=\{(x,y)\in\R^2\colon |x|+|y|\leqslant 1\}
\] 
and the simplex $S\subset\R^3$ with vertices $(-1,0,0)$, $(1,0,0)$, $(0,-1,1)$,  $(0,1,1)$. Observe that $S$ is a~graph of a~convex multifunction $F\colon D\to\nonempty(\R)$ (whose values are compact intervals) with no affine selection. Nevertheless, locally it is possible to put a~piece of a plane into $S$. It means that $F$ admits a~local affine selection at every $x_0\in\Int D$. We develop this observation in the next section.
\par\medskip
A.~Smajdor and W.~Smajdor proved in~\cite[Theorem~6]{SmaSma96} that if $F$ is defined on a~cone with the cone-basis in a~(real) vector space and $F$ takes the non-empty, closed (and necessarily convex) values in a~(real) locally convex space, then~$F$ admits an affine selection.

\section{Convex multifunctions with local selections}
Let $X$ be a topological vector space and let $D$ be a non-empty, convex subset of $X$ with non-empty interior. Moreover, let $Y$ be a real vector space. A multifunction $F\colon D\to\nonempty(Y)$ admits a~\emph{local affine selection} at a~point $x_0\in\Int D$, if there exist an open neighourhood $U\subset D$ of $x_0$ and an affine function $f\colon X\to Y$ such that $f(x)\in F(x)$ for every $x\in U$.
\par\medskip
The following finite-dimensional version of Edwards' theorem is quite easy to obtain.
\begin{lem}\label{lem:simplex}
 Let $S$ be an $n$-dimensional simplex in $\R^n$ and $Y$ be a~(real) vector space. Any convex multifunction $F:S\to\nonempty(Y)$ admits the affine selection.
\end{lem}
\begin{proof}
Let $a_0,\dots,a_n\in\R^n$ be the vertices of $S$ and let us choose $y_i\in F(a_i)$ ($i=0,\dots,n$). Then there exists (the unique) affine function $f\colon \R^n\to Y$ such that $f(a_i)=y_i$ ($i=0,\dots,n$). Take $x\in S$. Expressing $x$ as a~convex combination of $a_0,\dots,a_n$ (with coefficients $\lambda_0,\dots,\lambda_n \geqslant 0$, $\lambda_0+\ldots+\lambda_n = 1$) we get
 \[
  f(x)=f\left(\sum_{i=0}^n\lambda_ia_i\right)=\sum_{i=0}^n\lambda_if(a_i)=\sum_{i=0}^n\lambda_iy_i.
 \]
 Since $(a_i,y_i)\in \Gr F$, by convexity of this graph we arrive at
 \[
  \bigl(x,f(x)\bigr)=\left(\sum_{i=0}^n \lambda_ia_i,\sum_{i=0}^n \lambda_iy_i\right)=\sum_{i=0}^n\lambda_i(a_i,y_i)\in\Gr F,
 \]
 whence $f(x)\in F(x)$, $x\in X$ and the proof is complete.
\end{proof}
The above lemma allows us to prove the existence of local affine selections in the finite-dimensional case.
\begin{thm}
 Let $D\subset\R^n$ be a~convex set with a~non-empty interior and $Y$ be a~(real) vector space. Any convex multifunction $F\colon D\to\nonempty(Y)$ admits a~local affine selection at every interior point of~$D$.
\end{thm}
\begin{proof}
 Let $x_0\in\Int D$. Then $x_0$ is the interior point of some $n$-dimensional simplex $S\subset D$. By Lemma~\ref{lem:simplex} there exists the affine function $f:\R^n\to Y$ such that $f(x)\in F(x)$ for any $x\in S$. In particular, $f(x)\in F(x)$ for any $x\in U=\Int S$ and $f$ is a~desired local affine selection of~$F$.
\end{proof}

\section{Convex multifunctions without local selections}
In the infinite-dimensional case the problem of the existence of local affine selection looks completely different. Namely, there are convex multifunctions with no local affine selection. The following observations are due to Tomasz Kania (Warwick) who has kindly permitted us to include them here.\medskip

Let $X$ be a closed linear subspace of a Banach space $Y$. In the light of the Hahn--Banach theorem, the multifunction $F\colon X^*\to 2^{Y^*}$ given by
\begin{equation}\label{ro}F(f) = \{g\in Y^*\colon g|_X = f\text{ and }\|g\|=\|f\|\}\qquad (f\in X^*)\end{equation}
assumes always non-void values. Certainly, $F(f)$ is convex and weak*-compact for each $f\in X^*$. It is easy to prove that $F$ is a~convex multifunction.

\begin{prop}\label{pieszlasu}Suppose that $X$ is a closed linear subspace of a Banach space $Y$ such that $X^*$ does not embed isometrically into $Y^*$. Then $F$, as defined by \eqref{ro}, admits no local affine selection.\end{prop}
\begin{proof}Assume contrapositively that there exists an open neighbourhood $U$ of the origin such that $F$ admits an affine selection $\varphi$ say, when restricted to $U$. In particular, $\|\varphi(g)\|=\|g\|$ for all $g\in U$; thus $\varphi$ is isometric. Denote by $T$ the affine map $E^*\to F^*$ that extends $\varphi$. As $T0=0$, $T$ is a~linear, isometric embedding of $X^*$ into $Y^*$.\end{proof}

\begin{rem}The hypotheses of Proposition~\ref{pieszlasu} are easily met when $Y=C[0,1]$. Indeed, by the Banach--Mazur theorem, $C[0,1]$ contains isometric copies of all separable Banach spaces. The dual space of $C[0,1]$ is isometric to $L_1(\mu)$ for some measure $\mu$ and this, in turn, prevents many Banach spaces to embed into it (\emph{cf.}~\cite[proof of Proposition 4.3.8]{AlbKal06}).\medskip 

For instance take $X=\ell_1$. Then $X^*\cong\ell_\infty$ which contains isometrically all separable Banach spaces. In this case it is plain that any isometric copy of $\ell_1$ inside of $Y=C[0,1]$ meets the hypotheses of Proposition~\ref{pieszlasu}. \end{rem}

\begin{rem}When the hypotheses of Proposition~\ref{pieszlasu} are met, the mutlifunction is not lower semicontinuous. Indeed, otherwise by Lazar's theorem (Theorem~\ref{th:Lazar} in this note) it would have admitted an~affine selection. \end{rem}

Let $\beta \mathbb N$ denote the \v{C}ech--Stone compactification of the discrete space of natural numbers. Set
\begin{equation}\label{urysohn}F(f) = \{g\in C(\beta \mathbb N)\colon g|_{\beta \mathbb N\setminus \mathbb{N}} = f\text{ and } \|g\|=\|f\|\}\quad (f\in C(\beta \mathbb{N}\setminus \mathbb{N})).\end{equation}
Then by the Tietze--Urysohn theorem, $F(f)$ is non-empty for every $f\in C(\beta \mathbb{N}\setminus \mathbb{N})$. (It is also closed and convex.) Certainly the multifunction~$F$ is convex.
\begin{prop}The multifunction $F$ given by \eqref{urysohn} does not admit a local affine selection. \end{prop}
\begin{proof}Arguing as in the proof of Proposition~\ref{pieszlasu}, we would get a linear embedding of $C(\beta \mathbb{N}\setminus \mathbb{N})$ into $C(\beta \mathbb N)$, however it is known that no such operator exists. (For example, the former space does not have a strictly convex renorming but the latter does; the possibility of finding a strictly convex renorming passes to subspaces, \emph{cf.}~\cite{Bou80}; here we use the fact that $C(\beta \mathbb N\setminus \mathbb N)$ is isometric to $\ell_\infty / c_0$.) \end{proof}

\end{document}